\documentclass[12pt,twoside,final,psamsfonts]{amsart}
\usepackage[draft=false]{hyperref}
\usepackage[psamsfonts]{amssymb}
\usepackage[dvips,final]{graphicx,psfrag}
\usepackage{times,a4wide}

\usepackage[alphabetic,backrefs]{amsrefs}

\theoremstyle{plain}
\newtheorem*{theorem*}{Theorem}

\newtheorem{theorem}{Theorem}

\newtheorem*{proposition*}{Proposition}

\newtheorem*{corollary*}{Corollary}
\newtheorem{lemma}[theorem]{Lemma}
\newtheorem*{lemma*}{Lemma}

\theoremstyle{definition}

\newtheorem*{remark*}{Remark}
\newtheorem*{remarks*}{Remarks}
\newtheorem*{conjecture*}{Conjecture}
\theoremstyle{definition}

\newcommand{\D}{\mathbb{D}}
\newcommand{\C}{\mathbb{C}}
\newcommand{\B}{\mathbb{B}}

\newcommand{\N}{\mathbb{N}}

\newcommand{\E}{\mathbb{E}}

\newcommand{\Li}{\operatorname{Li}}
\newcommand{\Aut}{\operatorname{Aut}}

\newcommand{\Var}{\operatorname{Var}}

\renewcommand{\Im}{\operatorname{Im}}

\title[Volume fluctuations of random analytic varieties in the unit ball]
{Volume fluctuations of random analytic varieties in the unit ball}

\author[Xavier Massaneda] {Xavier Massaneda}

\address{Departament de Matem\`atica Aplicada i An\`alisi,
Universitat  de Bar\-ce\-lo\-na, Gran Via 585, 08071-Bar\-ce\-lo\-na, Spain}
\email{xavier.massaneda@ub.edu}

\author[Bharti Pridhnani] {Bharti Pridhnani}

\address{Departament de Matem\`atica Aplicada i An\`alisi,
Universitat  de Bar\-ce\-lo\-na, Gran Via 585, 08071-Bar\-ce\-lo\-na, Spain}
\email{bharti.pridhnani@ub.edu}

\thanks{Both authors supported by the Generalitat de Catalunya (grant 2009 SGR 01303) and the spanish Ministerio de Ciencia e Innovaci\'on 
(projects MTM2011-27932-C02-01).}

\date{\today}



\begin{document}

\begin{abstract} Given a Gaussian analytic function $f_L$ of intesity $L$ in the unit ball of $\C^n$, $n\geq 2$, consider its (random) zero variety $Z(f_L)$.  We study the variance of the $(n-1)$-dimensional volume of $Z(f_L)$ inside a pseudo-hyperbolic ball of radius $r$. We first express this variance as an integral of a positive function in the unit disk. Then we study its asymptotic behaviour as $L\to\infty$ and as $r\to 1^{-}$. Both the results and the proofs generalise to the ball those given by Jeremiah Buckley for the unit disk.
\end{abstract}

\maketitle

\section{Definitions and statements}

Let $\mathbb B_n$ denote the unit ball in $\mathbb C^n$ and let $\nu$ denote the Lebesgue measure in $\mathbb C^n$ normalised so that $\nu(\B_n)=1$. Explicitly $\nu=\frac{n!}{\pi^n} dm=\beta^n$, where $dm$ is the Lebesgue measure and $\beta=\frac {i}{2\pi}\partial\bar\partial |z|^2$ is the fundamental form of the Euclidean metric.

For $L>n$ consider the weighted Bergman space
\[
 B_L(\mathbb B_n)=\bigl\{f\in H(\B_n) : \|f\|_{n,L}^2:=c_{n,L} \int_{\B_n} |f(z)|^2 (1-|z|^2)^{L}  d\mu(z)< +\infty\bigr\} ,
\]
where  
\begin{equation}\label{m-invariant}
 d\mu(z)=\frac {d\nu(z)}{(1-|z|^2)^{n+1}},
\end{equation}
and $c_{n,L}=\frac{\Gamma(L)}{n!\Gamma(L-n)}$ is chosen so that $\|1\|_{n,L}=1$. 

Let
\[
 e_{\alpha}(z)=\left(\frac{\Gamma(L+|\alpha|)}{\alpha!\Gamma(L)}\right)^{1/2} z^\alpha
\]
denote the normalisation of the monomial $z^\alpha$ in the norm $\|\cdot\|_{n,L}$, so that $\{e_\alpha\}_\alpha$ is an orthonormal basis of $B_L(\mathbb B_n)$. As usual, here we denote $z=(z_1,\dots,z_n)$ and use the multi-index notation $\alpha=(\alpha_1,\dots,\alpha_n)$, $\alpha !=\alpha_1!\cdots \alpha_n!$, $|\alpha|=|\alpha_1|+\cdots +|\alpha_n|$ and $z^\alpha=z_1^{\alpha_1}\cdots z_n^{\alpha_n}$.

The \emph{hyperbolic Gaussian analytic function} (GAF) of \emph{intensity} $L$ is defined as
\[
 f_L(z)=\sum_\alpha a_\alpha  \left(\frac{\Gamma(L+|\alpha|)}{\alpha!\Gamma(L)}\right)^{1/2} z^\alpha \qquad z\in\B_n,
\]
where $a_\alpha $ are i.i.d. complex Gaussians of mean 0 and variance 1 ($a_\alpha \sim N_{\C}(0,1)$).


The sum defining $f_L$ can be analytically continued to $L>0$, which we assume henceforth.

The characteristics of the hyperbolic GAF are determined by its covariance kernel, which is given by (see \cite{ST1}*{Section 1}, \cite{Stoll}*{p.17-18})
\begin{align*}
 K_L(z,w)&=\E[f_L(z)\overline{f_L(w)}]= \sum_\alpha \frac{\Gamma(L+|\alpha|)}{\alpha!\Gamma(L)} z^\alpha \bar w^\alpha=
 \sum_{m=0}^\infty   \frac{\Gamma(L+m)}{\Gamma(L)} \sum_{\alpha : |\alpha|=m} \frac 1{\alpha!} z^\alpha \bar w^\alpha\\
 &=
 \sum_{m=0}^\infty \frac{\Gamma(L+m)}{m! \Gamma(L)} (z\cdot\bar w)^m=
 \frac 1{(1-z\cdot \bar w)^{L}}\ .
\end{align*}

A main feature of the hyperbolic GAF is that the distribution of the (random) integration current of its zero variety $Z_{f_L}$, 
\[
 [Z_{f_L}]=\frac {i}{2\pi}\partial\bar\partial\log |f_L|^2\ ,
\]
is invariant under automorphisms of the unit ball (see \cite{ST1}*{Section 1} or \cite{BMP}). Given $w\in\B_n$ there exists $\phi_w\in\Aut(\B_n)$ such that $\phi_w(w)=0$ and $\phi_w(0)=w$, and all automorphisms are essentially of this form: for all $\psi\in\Aut(B_n)$ there exist $w\in\B_n$ and $\mathcal U$ in the unitary group such that $\psi=\mathcal U \phi_w$ (see \cite{Ru}*{2.2.5}). Then the \emph{pseudo-hyperbolic distance} $\varrho$ in $\B_n$ is defined as
\[
 \varrho(z,w)=|\phi_w(z)|,\quad z,w\in\B_n\ ,
\]
and the corresponding pseudo-hyperbolic balls as 
\[
 E(w,r)=\{z\in\B_n : \varrho(z,w)<r\}, \qquad r<1\ .
\]

The Edelman-Kostlan formula (see \cite{HKPV}*{Section 2.4} and \cite{Sod}*{Theorem 1}) gives the so-called \emph{first intensity} of the GAF:
\[
\E[Z_{f_L}] =\frac{i}{2\pi}\partial \overline{\partial}\log K_L(z,z)=L\, \omega(z)\ ,
\]
where $\omega$ is the invariant form
\[
\omega(z)=\frac{i}{2\pi}\partial \overline{\partial}\log\bigl(\frac 1{1-|z|^2}\bigr)=\frac{1}{(1-|z|^2)^2}\sum^{n}_{j,k=1}[(1-|z|^2)\delta_{j,k}+z_k\overline{z_j}] \frac{i}{2\pi}dz_j\wedge d\overline{z_k}\ .
\]

In this paper we study the fluctuations of the $(n-1)$-dimensional volume of the random variety $Z_{f_L}$ inside a pseudo-hyperbolic ball 
$ E(z,r)$. By the invariance under automorphisms, this is equivalent to measuring the $(n-1)$-th volume of $Z_{f_L}$ inside $B(0,r)$. This volume is given by the integral
\[
 I_{f_L}(r)=\int_{B(0,r)\cap Z_{f_L}} \omega_{n-1}=\int_{B(0,r)} \frac {i}{2\pi}\partial\bar\partial\log |f_L|^2 \wedge \omega_{n-1}\ ,
\]
where for $p=1,\dots,n$, $\omega_p=\omega^p/p!$.

Notice that now $\omega^n=d\mu$, so the Edelman-Kostlan formula yields trivially
\[
 \mathbb E[I_{f_L}(r)]=L\int_{B(0,r)}\frac{\omega^n}{(n-1)!}=\frac L{(n-1)!}\mu(B(0,r))=\frac L{(n-1)!}\frac{r^{2n}}{(1-r^2)^n}\ .
\]

Our main goal is to study the variance
\[
 \Var I_{f_L}(r)= \E[(I_{f_L}(r)-\E(I_{f_L}(r)))^2]\ ,
\]
and, particularly, to describe its asymptotic behaviour as $L\to\infty$ and as $r\to 1^{-}$.

The computations are much simpler if we consider the Euclidean volume instead of the invariant volume defined above. Let
\[
 E_{f_L}(r)=\int_{B(0,r)\cap Z_{f_L}} \beta_{n-1}=\int_{B(0,r)} \frac {i}{2\pi}\partial\bar\partial\log |f_L|^2 \wedge \beta_{n-1}\ ,
\]
where, for $p=1,\dots,n$, $\beta_p=\beta^p/p!$.

The key result is the following reduction of $\Var E_{f_L}(r)$ to an integral of a positive function in the unit disk, together with the relation between $\Var I_{f_L}(r)$ and $\Var E_{f_L}(r)$.

\begin{theorem}\label{integral} Let $n\geq 2$. For $L>0$ and $r\in(0,1)$,

\begin{itemize}
 \item [(a)] $\displaystyle{\Var E_{f_L}(r)=\frac{r^{4n}  L^2 (1-r^2)^{2L-2}}{ (n-1)!(n-2)!} \int_{\D}\frac{(1-|w|^2)^{n-2}}{|1-r^2w|^{2L}-(1-r^2)^{2L}}\frac{|1-w|^2}{|1-r^2w|^2}\; \frac{dm(w)}{\pi}}$; 
 
 \item [(b)]  $\displaystyle{ \Var I_{f_L}(r)=\dfrac{\Var E_{f_L}(r)}{(1-r^2)^{2n-2}}}$.

\end{itemize}

\end{theorem}

From this we obtain the leading term in the asymptotics of $\Var I_{f_L}(r)$ as $L\to\infty$.

\begin{theorem}\label{th-L} Let $n\geq 2$ and fix $r\in (0,1)$. Then, as $L\to\infty$,
 \[ 
  \Var I_{f_L}(r)=\frac{1}{4\sqrt{\pi}}\frac{\zeta(n+1/2)}{(n-1)!} \frac{r^{2n-1}}{(1-r^2)^{n}}\frac 1{L^{n-3/2}}[1+o(1)]\ .
 \]
\end{theorem}

\begin{remarks*} 1. Even though the proof of Theorem~\ref{th-L} is only valid for $n\geq 2$, the statement matches with the corresponding result for  $n=1$ (see \cite[Theorem 2(a)]{Bu13} and the references therein). Notice also that $\Var E_{f_L}(r)=\Var I_{f_L}(r)$  when $n=1$.

2. As explained in \cite{SZ06}*{Section 2.2}, Theorem~\ref{th-L} can also be obtained with the methods used in the proof of the analogous result
in the context of compact manifolds. Let $p_N$ be a Gaussian holomorphic polynomial in $\C\mathbb P^n$ or, more generally, a section of a power $L^N$ of a positive Hermitian line bundle $L$ over an $n$-dimensional K\"ahler manifold $M$. Given a domain $\mathcal U\subset M$ with sufficiently regular boundary, define
\[
 A_{p_N}(\mathcal U)=\int_{Z_{p_N}\cap\, \mathcal U}\frac{\omega^{n-1}}{(n-1)!}\ ,
\]
where $\omega$ denotes the K\"ahler form of $M$.  According to B. Shiffman and S. Zelditch \cite[Theorem 1.4.]{SZ08} (with $k=1$) 
\[
 \Var A_{p_N}(\mathcal U)=\frac 1{N^{n-3/2}}\left[\frac{\pi^{n-5/2}}8 \zeta(n+1/2) \sigma(\partial\mathcal U)+O(N^{\epsilon-1/2})\right]\ ,
\]
where $\sigma(\partial\mathcal U)$ denotes the $(2n-1)$-volume of the boundary $\partial\mathcal U$. Notice that $\frac{r^{2n-1}}{(1-r^2)^{n}}$ is the $(2n-1)$-volume (with respect to the invariant form) of $\partial B(0,r)$.

The proof of this result is based on a (pluri)bipotential expression of $\Var I_{f_L}(r)$ (see the beginning os Section 2) and on good estimates of the covariance kernel, something we certainly have for the hyperbolic GAF in the ball.

3. Theorem~\ref{th-L} shows a strong form of ``self-averaging'' of the volume $I_{f_L}(r)$, in the sense that the fluctuations are of smaller order than the expected value (see remarks also in \cite{SZ08}). More precisely,as $L\to\infty$
\[
 \frac{\Var I_{f_L}(r)}{(\E[I_{f_L}(r)])^2}=\textrm O\left(\frac 1{L^{n+1/2}}\right)\ .
\]
Notice also that the rate of self-averaging increases with the dimension.

\end{remarks*}

We also study the behaviour of $\Var I_{f_L}(r)$ as $r\to 1^{-}$.

\begin{theorem}\label{th-r} Let $n\geq 2$ and fix $L>0$. Then, as $r\to 1^{-}$:
\begin{itemize}
 \item [(a)] If $L<n/2$ then, 
 \[
 \Var I_{f_L}(r)=C(L,n)\frac{1}{(1-r^2)^{2(n-L)}}[1+o(1)]
 \]
 where
 \[
  C(L,n)=\frac{L^2}{\sqrt{\pi}}\frac{2^{n-1}}{4^L (n-1)!}\dfrac{\Gamma(\frac n2-L) \Gamma(\frac{n+1}2-L)}{(\Gamma(n-L))^2} .
 \]

 \item[(b)] If $L=n/2$ then,
 \[
 \Var I_{f_L}(r)=C(n/2,n)\frac{1}{(1-r^2)^{n}}\log\bigl(\frac 1{1-r^2}\bigr)[1+o(1)]
 \]
 where
 \[
  C(n/2,n)=\frac{(n/2)^2}{(n-1)!\bigl(\Gamma(\frac n2)\bigr)^2}\ .
 \]
 
  \item[(c)] If $L>n/2$ then, 
 \[
 \Var I_{f_L}(r)=C(L,n)\frac{1}{(1-r^2)^{n}}[1+o(1)]
 \]
 where
 \[
  C(L,n)=\frac{L^2}{4\sqrt{\pi} (n-1)!}\sum_{k=1}^\infty \frac{\Gamma(Lk-\frac n2) \Gamma(Lk-\frac {n-1}2)}{(\Gamma(Lk+1))^2} 
  \bigl(Lk+\dfrac{n(n-1)}2\bigr)\ .
 \]
\end{itemize}

\end{theorem}

\begin{remarks*}
 (a) Notice that for $L$ fixed and $r\to 1^{-}$ there is also a strong self-averaging of the volume $I_{f_L}(r)$, which also increases with the dimension: 
 \[
   \frac{\Var I_{f_L}(r)}{(\E[I_{f_L}(r)])^2}=
   \begin{cases}
    \textrm O\left((1-r^2)^{2L}\right)\quad & L<n/2 \\
    \textrm O\left((1-r^2)^n\log\bigl(\frac 1{1-r^2}\bigr)\right) \quad &L=n/2\\
    \textrm O\left((1-r^2)^n\right) \quad &L>n/2\ .
   \end{cases}
 \]

 (b) As before, the proofs we give are valid only for $n\geq 2$ but the values $C(n,L)$ above match those obtained by J. Buckley in the disk \cite[Theorem 1]{Bu13} (since $1-r^2=2(1-r)+o(1-r)$).

 (c) As in dimension 1, there is a change of regime at $L=n/2$, which we don't know how to explain.
 
 (d) Using the asymptotics $\lim\limits_{k\to\infty}\frac{\Gamma(k+a)}{\Gamma(k) k^a}=1$ we see that as $L\to\infty$
 \begin{align*}
  C(L,n)&=\frac{L^2}{4\sqrt{\pi} \Gamma(n)}\left(\sum_{k=1}^\infty\frac 1{(Lk)^{n+1/2}}\right) (1+\textrm{o}(1))
  = \frac{1}{4\sqrt{\pi} \Gamma(n)} \frac{\zeta(n+1/2)}{L^{n-3/2}} (1+\textrm{o}(1))\,
 \end{align*}
 and in particular, by Theorem~\ref{th-L}, the limits in $r$ and $L$ can be interchanged:
 \begin{align*}
  \lim_{L\to\infty}\lim_{r\to 1^{-}} L^{n-3/2} (1-r^2)^{n} \Var I_{f_L}(r)&=
 \lim_{r\to 1^{-}} \lim_{L\to\infty} L^{n-3/2} (1-r^2)^{n} \Var I_{f_L}(r)\\
& = \frac{\zeta(n+1/2)}{4\sqrt{\pi} \Gamma(n)}\ .
 \end{align*}

\end{remarks*}

The scheme of the proofs of Theorems \ref{integral}, \ref{th-L}, and \ref{th-r}
is the same as in the one dimensional case (see \cite{Bu13}), although some of the computations are considerably more involved. 

The paper is structured as follows. Section 2 contains the proof of Theorem~\ref{integral}, which is a long and sometimes tricky computation. In Section 3 we prove Theorem~\ref{th-L}, while Section 4 is devoted to the proof of Theorem~\ref{th-r}. 

\section{Proof of the Theorem~\ref{integral}}

(a) As in dimension $n=1$, or as in the context of compact manifolds, the variance we want to study can be expressed through the bipotential. Denote $S(0,r)=\{\zeta\in\C^n : |\zeta|=r\}$ and $S=S(0,1)$. Then (see \cite[Section 2]{Bu13}, \cite[Theorem 3.11]{SZ08})
\begin{multline*}
 \Var  E_{f_L}(r)= 
 \int\limits_{S(0,r)} \int\limits_{S(0,r)} \sum_{j, k=1}^n\frac{\partial^2 \rho_L}{\partial \bar z_j \partial\bar w_k}(z,w)\; 
 \frac i{2\pi} d\bar z_j\wedge \beta_{n-1}(z) \wedge \frac i{2\pi}d\bar w_k \wedge \beta_{n-1}(w) \ .
\end{multline*}
Here $\rho_L(z,w)=\Li((\theta(z,w))^L)$, where $\Li(x)=\sum_{m=1}^\infty \frac {x^m}{m^2}$ and
\[
 \theta(z,w)=\frac{(1-|z|^2)(1-|w|^2)}{|1-z\cdot\bar w|^2}\ .
\]
Notice that $\theta^{L/2}(z,w)$ coincides with the modulus of the normalised kernel of the GAF.

For the sake of completeness we sketch the proof of this identity given in \cite{SZ06}*{Proposition 3.5}. Let $U=B(0,r)$. By the Edelman-Kostlan formula
\[
 E_{f_L}(r)-\E[E_{f_L}(r)]=\int_U\frac i{2\pi}\partial\bar\partial \log |\hat f_L|^2\wedge\beta_{n-1}=\int_{\partial U}\frac i{2\pi}\bar\partial \log|\hat f_L|^2\wedge\beta_{n-1}\ ,
\]
where
\[
 \hat f_L(z)=\frac{f_L(z)}{\sqrt{K_L(z,z)}}
\]
is the normalised GAF. Then
\begin{align*}
 \Var(E_{f_L}(r))&=\E\left[(E_{f_L}(r)-\E[E_{f_L}(r)])^2\right]\\
 &=\E\left[\int_{\partial U}\frac i{2\pi}\bar\partial \log|\hat f_L(z)|^2\wedge\beta_{n-1}(z) \int_{\partial U}\frac i{2\pi}\bar\partial \log|\hat f_L(w)|^2\wedge\beta_{n-1}(w)\right]\\
 &=\int_{\partial U}\int_{\partial U} \E\left[\frac i{2\pi}\bar\partial_z \log|\hat f_L(z)|^2\wedge \frac i{2\pi}\bar\partial_w \log|\hat f_L(w)|^2\right]\wedge\beta_{n-1}(z)\wedge\beta_{n-1}(w)\\
 &=\int_{\partial U}\int_{\partial U}\left(\frac i{2\pi}\right)^2 \bar\partial_z \bar\partial_w \E[\log|\hat f_L(z)|^2 \log|\hat f_L(w)|^2]\wedge\beta_{n-1}(z)\wedge\beta_{n-1}(w)\ .
\end{align*}
The result follows from the fact that (see for instance \cite{HKPV}*Lemma 3.5.2)
\[
 \E[\log|\hat f_L(z)|^2 \log|\hat f_L(w)|^2]=\rho_L(z,w)\ \boxtimes\ .
\]

To compute the  integrals above we use that $\beta_n=\bigwedge_{j=1}^n \frac i{2\pi} dz_j\wedge d\bar z_j$. For $j=1,\dots,n$ define the $(n-1,n)$-forms
\begin{equation}\label{gamma}
 \gamma_j(z)=\frac i{2\pi} d\bar z_j\wedge\bigwedge_{k\neq j} \frac i{2\pi} dz_k\wedge d\bar z_k= \frac i{2\pi} d\bar z_j\wedge \beta_{n-1}(z)\ .
\end{equation}

Denoting $S=S(0,1)$ and letting  $z=r\xi$, $w=r\eta$, where $\xi,\eta\in S$ we have, from the expression above,
\begin{equation}\label{iintS}
 \Var E_{f_L}(r)= \int_S \int_S \sum_{j,k=1}^n \frac{\partial^2 \rho_L}{\partial\bar z_j \partial\bar w_k} (r\xi,r\eta)\gamma_j(r\xi) \gamma_k(r\eta)\ .
\end{equation}

\begin{lemma}\label{laplacian-rho} Let $r\in(0,1)$ and $\xi, \eta\in S$. Then
\begin{multline*}
 \frac{\partial^2 \rho_L}{\partial\bar z_j \partial\bar w_k} (r\xi,r\eta)=\frac{L^2(1-r^2)^{2L-2}r^2}{|1-r^2\bar \xi\cdot\eta|^{2L}-(1-r^2)^{2L}}\times\\
\times \frac{[(1-r^2)\eta_j-\xi_j(1-r^2\bar \xi\cdot \eta)] [(1-r^2) \xi_k-\eta_k(1- \xi\cdot \bar \eta)]}{|1-r^2\bar \xi\cdot \eta|^2}\ .
\end{multline*}

\end{lemma}

\begin{proof}
Since $\Li'(x)=\frac 1x\log\left(\frac 1{1-x}\right)$, we have
 \begin{align*}
  \frac{\partial \rho_L}{\partial\bar z_j}&=\Li'(\theta^L)L\theta^{L-1} \frac{\partial \theta}{\partial\bar z_j}=\frac 1{\theta^L} \log\left(\frac 1{1-\theta^L}\right) L\theta^{L-1} \frac{\partial \theta}{\partial\bar z_j}\\
  &=\frac{L}{\theta} \log\left(\frac 1{1-\theta^L}\right) \frac{\partial \theta}{\partial\bar z_j}
 \end{align*}
 and
 \begin{align*}
  \frac{\partial^2 \rho_L}{\partial\bar z_j \partial\bar w_k}&=-\frac{L}{\theta^2} \frac{\partial \theta}{\partial\bar w_k} \log\left(\frac 1{1-\theta^L}\right)\frac{\partial \theta}{\partial\bar z_j}+ \frac{L}{\theta}\frac{L\theta^{L-1}}{1-\theta^L} \frac{\partial \theta}{\partial\bar w_k} \frac{\partial \theta}{\partial\bar z_j} + \frac{L}{\theta} \log\left(\frac 1{1-\theta^L}\right) \frac{\partial^2 \theta}{\partial\bar z_j \partial\bar w_k}\\
  &=\frac{L}{\theta} \log\left(\frac 1{1-\theta^L}\right)\left[\frac{\partial^2 \theta}{\partial\bar z_j \partial\bar w_k}-\frac 1{\theta} \frac{\partial \theta}{\partial\bar z_j}\frac{\partial \theta}{\partial\bar w_k}\right]+\frac{L^2}{\theta^2}\frac{\theta^L}{1-\theta^L} \frac{\partial \theta}{\partial\bar z_j}\frac{\partial \theta}{\partial\bar w_k}\ .
 \end{align*}
 Using the definition of $\theta$ given above,
 \[
  \frac{\partial \theta}{\partial\bar z_j}=\frac{-z_j(1-|w|^2)}{|1-\bar z\cdot w|^2}+\frac{(1-|z|^2)(1-|w|^2)}{(1- z\cdot \bar w)(1-\bar z\cdot w)^2} w_j=
  \frac{1-|w|^2}{|1-\bar z\cdot w|^2}\left(\frac{1-|z|^2}{1-\bar z\cdot w}w_j-z_j\right)\ .
 \]
We deduce that
 \[
  \frac{\partial^2 \theta}{\partial\bar z_j \partial\bar w_k}-\frac 1{\theta} \frac{\partial \theta}{\partial\bar z_j}\frac{\partial \theta}{\partial\bar w_k}=0
 \]
and therefore, 
 \begin{align*}
  \frac{\partial^2 \rho_L}{\partial\bar z_j \partial\bar w_k}&=\frac{L^2}{\theta^2} \frac{\theta^L}{1-\theta^L} \frac{\partial \theta}{\partial\bar z_j}\frac{\partial \theta}{\partial\bar w_k}\\
  &=\frac{L^2}{\theta^2} \frac{\theta^L}{1-\theta^L} \frac{(1-|z|^2)(1-|w|^2)}{|1-\bar z\cdot w|^4} \left(\frac{1-|z|^2}{1-\bar z\cdot w}w_j-z_j\right)\left(\frac{1-|w|^2}{1- z\cdot\bar w} z_k-w_k\right).
  \end{align*}
Substituting $z=r\xi$, $w=r\eta$, and using the identity
\[
 \frac{\theta^L(r\xi,r\eta)}{1-\theta^L(r\xi,r\eta)}=\frac{(1-r^2)^{2L}}{|1-r^2\bar \xi\cdot \eta|^{2L}}\frac 1{1-\frac{(1-r^2)^{2L}}{|1-r^2\bar \xi\cdot \eta|^{2L}}}=\frac{(1-r^2)^{2L}}{|1-r^2\bar \xi\cdot \eta|^{2L}-(1-r^2)^{2L}}
\]
we get the result.
\end{proof}

Plugging this into \eqref{iintS}  we finally have
\begin{equation}\label{reduction}
\Var I_{f_L}(r)=  L^2 (1-r^2)^{2(L-1)} r^2 \mathcal I(L,r)\ ,
\end{equation}
where
\[
\mathcal I(L,r)=\int_S \int_S \frac{1}{|1-r^2\bar \xi\cdot \eta|^{2L}-(1-r^2)^{2L}} \frac{\Omega(r\xi, r\eta)}{|1-r^2\bar \xi\cdot \eta|^{2}}
\]
and $\Omega(r\xi, r\eta)$ is the $(n-1,n-1)$-form (in $\xi$ and $\eta$) given by
\[
 \Omega(r\xi, r\eta)=\sum_{j,k=1}^n [(1-r^2)\eta_j-\xi_j(1-r^2\bar \xi\cdot \eta)][(1-r^2)\xi_k-\eta_k(1- r^2\xi\cdot \bar \eta)] \gamma_j(r\xi)\gamma_k(r\eta).
\]

We operate
\begin{multline*}
 \bigl[(1-r^2)\eta_j-\xi_j(1-r^2\bar \xi\cdot \eta)\bigr]\bigl[(1-r^2)\xi_k-\eta_k(1- r^2\xi\cdot \bar \eta)\bigr]=\\
 =(1-r^2)^2\xi_k\eta_j-(1-r^2) (1-r^2\bar \xi\cdot \eta) \xi_j\xi_k- (1-r^2) (1-r^2 \xi\cdot \bar\eta) \eta_j \eta_k+|1-r^2\bar \xi\cdot \eta|^2\xi_j\eta_k
\end{multline*}
and split $\mathcal I(L,r)=\sum_{m=1}^4 \mathcal I_m(L,r)$, where
\begin{align*}
 \mathcal I_1(L,r)&=\int_S \int_S \frac {(1-r^2)^2}{|1-r^2\bar \xi\cdot \eta|^{2L}-(1-r^2)^{2L}}
 \frac {1}{|1-r^2\bar \xi\cdot \eta|^{2}} \left(\sum_{j=1}^n\eta_j \gamma_j(r\xi)\right)\left(\sum_{k=1}^n \xi_k\gamma_k(r\eta)\right)\\
\mathcal   I_2(L,r)&=-\int_S \int_S \frac {1-r^2}{|1-r^2\bar \xi\cdot \eta|^{2L}-(1-r^2)^{2L}}
 \frac {1}{1-r^2 \xi\cdot \bar\eta} \left(\sum_{j=1}^n\xi_j \gamma_j(r\xi)\right)\left(\sum_{k=1}^n \xi_k\gamma_k(r\eta)\right)\\
\mathcal   I_3(L,r)&=-\int_S \int_S \frac {1-r^2}{|1-r^2\bar \xi\cdot \eta|^{2L}-(1-r^2)^{2L}}
 \frac {1}{1-r^2 \bar\xi\cdot \eta} \left(\sum_{j=1}^n\eta_j \gamma_j(r\xi)\right)\left(\sum_{k=1}^n \eta_k\gamma_k(r\eta)\right)\\
\mathcal  I_4(L,r)&= \int_S \int_S \frac {1}{|1-r^2\bar \xi\cdot \eta|^{2L}-(1-r^2)^{2L}}
 \left(\sum_{j=1}^n\xi_j \gamma_j(r\xi)\right)\left(\sum_{k=1}^n \eta_k\gamma_k(r\eta)\right)\ .
\end{align*}

In order to compute these integrals we first fix $\eta$ and consider a unitary transformation $\mathcal U$ taking $e_1=(1,0,\dots,0)$ to $\eta$. Denote $v^1=\eta$ and $v^j=\mathcal U(e_j)$, $j>1$, where $e_j=(0,\dots,\stackrel{\stackrel{j}{\smile}}{1},\dots,0)$, so that $\eta, v^2, \dots,v^n$ is an orthonormal system. 

Consider the change of variables $\xi=\mathcal U(\alpha)=\sum_{j=1}^n \alpha_j v^j$. Then 
\[
 \xi\cdot\bar\eta=\mathcal U(\alpha)\cdot \overline{\mathcal U(e_1)}=\alpha\cdot \bar e_1=\alpha_1\ .
\]
Also
\[
 \xi_k=\sum_{m=1}^n \alpha_m v^m_k\qquad ,\qquad \bar\xi_k=\sum_{j=1}^n \bar\alpha_m \bar v^m_k\ ,
\]
and therefore
\[
 d\xi_k=\sum_{m=1}^n v^m_k d\alpha_m \qquad  ,\qquad  d\bar\xi_k=\sum_{m=1}^n \bar v^m_k d\bar\alpha_m \ .
\]

Since $\beta(\xi)=\frac i{2\pi}\partial\bar\partial |\xi|^2$ is invariant by unitary transformations
\begin{equation}\label{change-gamma}
  \gamma_j(r\xi)=\frac i{2\pi}r d\bar\xi_j\wedge \beta_{n-1}(r\xi)=
  \frac i{2\pi} \left(r\sum_{m=1}^n \bar v^m_j d\bar\alpha_j\right)\wedge
  \beta_{n-1}(r\alpha)=\sum_{m=1}^n \bar v^m_j \gamma_m(r\alpha)\ .
\end{equation}

Now parametrise $\alpha=\mathcal U^{-1}(\xi)\in S$ with coordinates $w=(w_1,\dots,w_{n-1})\in\B_{n-1}$, $\psi\in[0,2\pi)$ in the following way
\[
 \begin{cases}
  \alpha_j=w_j\qquad \qquad &\ j=1,\dots,n-1\\
  \alpha_n=\sqrt{1-|w|^2} e^{i\psi}\ .&
 \end{cases}
\]

Let us write the forms $\gamma_j(r\alpha)$ in this parametrisation.

\begin{lemma}\label{change} Let $d\beta_{n-1}(w)=\bigwedge_{k=1}^{n-1}\frac i{2\pi} dw_k\wedge d\bar w_k$. Then
 \[  
\begin{cases}
  \gamma_j(\alpha)&=\bar w_j \dfrac{d\psi}{2\pi} \wedge d\beta_{n-1}(w)\qquad j=1,\dots,n-1,\\
  \gamma_n(\alpha)&= \sqrt{1-|w|^2} e^{-i\psi}\dfrac{d\psi}{2\pi}\wedge d\beta_{n-1}(w)
 \end{cases}
 \]
\end{lemma}

\begin{proof}
Directly from the definition we have
\begin{align*}
 d\alpha_n&=\sum_{l=1}^{n-1}\left(\frac{-\bar w_l e^{i\psi}}{2\sqrt{1-|w|^2}} d w_l+ \frac{- w_l e^{i\psi}}{2\sqrt{1-|w|^2}} d \bar w_l\right)+\sqrt{1-|w|^2} ie^{i\psi} d\psi\\
 d\bar\alpha_n&=\sum_{l=1}^{n-1}\left(\frac{-\bar w_l e^{-i\psi}}{2\sqrt{1-|w|^2}} d w_l+ \frac{- w_l e^{-i\psi}}{2\sqrt{1-|w|^2}} d \bar w_l\right)-\sqrt{1-|w|^2} ie^{-i\psi} d\psi\ .
\end{align*}

Assume first that $j<n$. Then  $\bigwedge_{k\neq j}\frac i{2\pi} d\alpha_k\wedge d\bar \alpha_k$ contains the factor $\frac i{2\pi}  d\alpha_n\wedge d\bar \alpha_n$, so
\begin{multline*}
 \bigwedge_{k\neq j}\frac i{2\pi}  d\alpha_k\wedge d\bar \alpha_k=\bigwedge_{k\neq j,n}\frac i{2\pi}  d\alpha_k\wedge d\bar \alpha_k\wedge \frac i{2\pi}  d\alpha_n\wedge d\bar \alpha_n\\
 =\bigwedge_{k\neq j}\frac i{2\pi}  d w_k\wedge d\bar w_k\wedge \frac i{2\pi} \left(\frac{-\bar w_j e^{i\psi}}{2\sqrt{1-|w|^2}} d w_j+ \frac{- w_j e^{i\psi}}{2\sqrt{1-|w|^2}} d \bar w_j+\sqrt{1-|w|^2} ie^{i\psi} d\psi\right)\wedge\\
 \wedge\left(\frac{-\bar w_l e^{-i\psi}}{2\sqrt{1-|w|^2}} d w_l+ \frac{- w_l e^{-i\psi}}{2\sqrt{1-|w|^2}} d \bar w_l-\sqrt{1-|w|^2} ie^{-i\psi} d\psi\right)\\
 =\frac 1{2\pi}  d\psi\wedge(\bar w_j \, dw_j +  w_j \, d\bar w_j)\wedge \bigwedge_{k\neq j}\frac i{2\pi}  d w_k\wedge d\bar w_k\ .
\end{multline*}
Therefore
\begin{align*}
 \gamma_j(\alpha)&=\frac i{2\pi}  d\bar\alpha_j\wedge \bigwedge_{k\neq j}\frac i{2\pi}  d \alpha_k\wedge d\bar \alpha_k\\
 &=\frac i{2\pi}  d\bar w_j\wedge\frac 1{2\pi} d\psi\wedge(\bar w_j \, dw_j +  w_j \, d\bar w_j)\wedge \bigwedge_{k\neq j}\frac i{2\pi}  d w_k\wedge d\bar w_k\\
 &=\bar w_j \frac{d\psi}{2\pi}\wedge \bigwedge_{k=1}^{n-1}\frac i{2\pi}  d w_k\wedge d\bar w_k
\end{align*}

Assume now that $j=n$. Then
\begin{align*}
  \gamma_n(\alpha)&=\frac i{2\pi}  d\bar\alpha_n\wedge \bigwedge_{k<n}\frac i{2\pi}  d \alpha_k\wedge d\bar \alpha_k= 
\sqrt{1-|w|^2}  e^{-i\psi}\; \frac{d\psi}{2\pi}\wedge d\beta_{n-1}(w)
\end{align*}
\end{proof}

We finally use the parametrisation of Lemma~\ref{change} to compute the integrals $\mathcal I_j(L,r)$. We begin with $\mathcal I_4$.

\underline{ $\mathcal I_4(L,r)$}. 
First make the change of variables $\xi=\mathcal U(\alpha)$.
By invariance by unitary transformations,
\[
\sum_{j=1}^n \xi_j \gamma_j(r\xi)=\frac i{2\pi}\bar\partial |\xi|^2\wedge \beta_{n-1}(r\xi)=\sum_{j=1}^n \alpha_j \gamma_j(r\alpha),
\]
and therefore
\begin{align*}
 \mathcal I_4(L,r)&=\int_S\int_S \frac{1}{|1-r^2\alpha_1|^{2L}-(1-r^2)^{2L}}\left(\sum_{j=1}^n \alpha_j \gamma_j(r\alpha)\right) \left(\sum_{k=1}^n \eta_k \gamma_k(r\eta)\right)\\
 &=A_n \int_S \frac{1}{|1-r^2\alpha_1|^{2L}-(1-r^2)^{2L}}\left(\sum_{j=1}^n \alpha_j \gamma_j(r\alpha)\right)\ ,
\end{align*}
where, by Stokes and the identity $\beta^n=d\nu$,
\begin{equation}\label{an}
 A_n:=\int_S \sum_{k=1}^n \eta_k \gamma_k(r\eta)=r^{2n-1} \int_S \frac i{2\pi}\bar\partial |\eta|^2\wedge \beta_{n-1}(\eta)=r^{2n-1} 
 \int_{\B_n} \frac{\beta^n(z)}{(n-1)!}=\frac{r^{2n-1}}{(n-1)!}\ .
\end{equation}

We compute the integral in $\alpha$ 
using the parametrisation of Lemma~\ref{change}. Since $\gamma_j(r\alpha)=r^{2n-1}\gamma_j (\alpha)$ and
\begin{align}\label{BN}
 \sum_{j=1}^n \alpha_j \gamma_j(\alpha)&=\sum_{j=1}^{n-1} w_j\bar w_j \frac{d\psi}{2\pi}\wedge d\beta_{n-1}(w)
 +\sqrt{1-|w|^2} e^{i\psi} \sqrt{1-|w|^2}e^{-i\psi}\, \frac{d\psi}{2\pi}\wedge d\beta_{n-1}(w)  \nonumber \\
 &=(|w|^2+1-|w|^2) \, \frac{d\psi}{2\pi}\wedge d\beta_{n-1}(w)= \frac{d\psi}{2\pi}\wedge d\beta_{n-1}(w), 
\end{align}
we have, after integrating $w_2,\dots,w_{n-1},\psi$,
\begin{align*}
  \mathcal I_4(L,r)&=A_n \int_{\B_{n-1}}\int_0^{2\pi}\frac {r^{2n-1}}{|1-r^2 w_1|^{2L}-(1-r^2)^{2L}}\,  \frac{d\psi}{2\pi}\wedge d\beta_{n-1}(w)\\
  &= A_n  \int_{\D}\frac  {r^{2n-1}}{|1-r^2 w_1|^{2L}-(1-r^2)^{2L}}\frac{1}{(n-2)!} (1-|w_1|^2)^{n-2} \frac{dm(w_1)}{\pi}\\
  &=\frac{A_n r^{2n-1}}{(n-2)!}\int_{\D}\frac {(1-|w_1|^2)^{n-2}}{|1-r^2 w_1|^{2L}-(1-r^2)^{2L}} \frac{dm(w_1)}{\pi}\ .
\end{align*}

\underline{$\mathcal I_1(L,r)$}. As in the previous case, first we change $\xi=\mathcal U(\alpha)$. By \eqref{change-gamma}, and since $\eta=v^1$
and the system $\{v^l\}_{l=1}^n$ is orthonormal, the form to integrate is
\begin{align*}
 \Gamma_1&:=\sum_{j,k=1}^n \eta_j\xi_k \gamma_j(r\xi)\gamma_k(r\eta)=\sum_{j,k=1}^n \eta_j\left(\sum_{l=1}^n \alpha_l v_k^l\right)\left(\sum_{m=1}^n\bar v_j^m \gamma_m(r\alpha)\right) \gamma_k(r\eta)\\
 &=\sum_{k,l,m=1}^n \left(\sum_{j=1}^n v_j^l \bar v_j^m\right) \alpha_l v_k^l \gamma_m(r\alpha) \gamma_k(r\eta)
=\sum_{k=1}^n \left(\sum_{l=1}^n \alpha_l v_k^l\right) \gamma_1(r\alpha) \gamma_k(r\eta)\ .
\end{align*}

Now we use the parametrisation given in
Lemma~\ref{change}. Then
\begin{align*}
 \sum_{l=1}^n \alpha_l v_k^l&=w_1\eta_k+\sum_{l=2}^{n-1} w_l v_k^l +\sqrt{1-|w|^2} e^{i\psi} v_k^n\ ,
\end{align*}
and therefore
\begin{align*}
 \left(\sum_{l=1}^n \alpha_l v_k^l\right) \gamma_1(r\alpha)&=r^{2n-1}\left(w_1\eta_k+\sum_{l=2}^{n-1} w_l v_k^l +\sqrt{1-|w|^2} e^{i\psi} v_k^n\right)\bar w_1 \frac{d\psi}{2\pi}\wedge d\beta_{n-1}(w)\\
 &=r^{2n-1}\left(|w_1|^2 \eta_k +\sum_{l=2}^{n-1} \bar w_1 w_l v_k^l + \bar w_1 \sqrt{1-|w|^2} e^{i\psi} v_k^n\right) \frac{d\psi}{2\pi}\wedge d\beta_{n-1}(w)\ .
\end{align*}

This form will be integrated against a function which does not depend on $w_2,\dots,w_{n-1},\psi$. The
last term in the sum will vanish when integrating in $\psi$, while the second term will vanish when integrating in $w_2,\dots,w_n$. Thus, in terms of the integration we want to perform,
\begin{align*}
 \Gamma_1&=r^{2n-1}\sum_{k=1}^n |w_1|^2 \eta_k\; \frac{d\psi}{2\pi}\wedge d\beta_{n-1}(w)\gamma_k(\eta)+ \textrm{vanishing terms} \\
 &=r^{2n-1}\left(\sum_{k=1}^n\eta_k \gamma_k(\eta)\right)  |w_1|^2\; \frac{d\psi}{2\pi}\wedge d\beta_{n-1}(w)+ \textrm{vanishing terms}\ .
\end{align*}
Letting $A_n$ be the constant \eqref{an}, and integrating $w_2,\dots,w_{n-1},\psi$, we obtain
\begin{align*}
 \mathcal I_1(L,r)&=A_n\int_{\B_{n-1}}\int_0^{2\pi}\frac {r^{2n-1}}{|1-r^2w_1|^{2L}-(1-r^2)^{2L}}\frac{(1-r^2)^{2}}{|1-r^2w_1|^{2}} |w_1|^2 \frac{d\psi}{2\pi}\wedge d\beta_{n-1}(w)\\
 &=\frac{r^{2n-1} A_n}{(n-2)!}\int_{\D}\frac {(1-|w_1|^2)^{n-2}}{|1-r^2w_1|^{2L}-(1-r^2)^{2L}}\frac{(1-r^2)^{2}}{|1-r^2w_1|^{2}}|w_1|^2 \frac{dm(w_1)}{\pi}
\end{align*}

\underline{$\mathcal I_2(L,r)$}. Once more, changing $\xi=\mathcal U(\alpha)$ and parametrising as in Lemma~\ref{change},
\begin{align*}
 \Gamma_2&:=\sum_{j,k=1}^n \xi_j\xi_k \gamma_j(r\xi)\gamma_k(r\eta)=\sum_{j,k=1}^n \left(\sum_{l=1}^n \alpha_l v_j^l\right)
\left(\sum_{m=1}^n \alpha_m v_k^m\right) \left(\sum_{t=1}^n \bar v_j^t \gamma_t(r\alpha)\right)\gamma_k(r\eta)\\
&=\sum_{k,l,m,t=1}^n \alpha_l \alpha_m v_k^m \left(\sum_{j=1}^n v_j^l  \bar v_j^t\right)\gamma_t(r\alpha)\gamma_k(r\eta)
=\left(\sum_{l=1}^n \alpha_l \gamma_l(r\alpha)\right)\left(\sum_{k,m=1}^n \alpha_m v_k^m \gamma_k(r\eta)\right)\ .
 \end{align*}
Hence
\begin{align*}
 \sum_{k,m=1}^n \alpha_m v_k^m \gamma_k(r\eta)&=\sum_{m=1}^{n-1} w_m\left(\sum_{k=1}^n v_k^m \gamma_k(r\eta)\right)+\sqrt{1-|w|^2} e^{i\psi}\left(\sum_{k=1}^n v_k \gamma_k (r\eta)\right)\\
 &=w_1 \sum_{k=1}^n \eta_k \gamma_k(r\eta)+ \sum_{m=2}^{n-1} w_m\left(\sum_{k=1}^n v_k^m \gamma_k(r\eta)\right)+\sqrt{1-|w|^2} e^{i\psi}\left(\sum_{k=1}^n v_k \gamma_k (r\eta)\right)
\end{align*}
and therefore, using \eqref{BN},
\begin{multline*}
 \Gamma_2=r^{2n-1}\left[w_1 \sum_{k=1}^n \eta_k \gamma_k(r\eta)+ \sum_{m=2}^{n-1} w_m\left(\sum_{k=1}^n v_k^m \gamma_k(r\eta)\right)+\right. \\
 \left. +\sqrt{1-|w|^2} e^{i\psi}\left(\sum_{k=1}^n v_k \gamma_k (r\eta)\right)\right] \frac{d\psi}{2\pi}\wedge d\beta_{n-1}(w)\ .
 \end{multline*}
As before, the third term in the bracket will vanish when integrating in $\psi$, while the second term will vanish when integrating in $w_2,\dots, w_{n-1}$. Thus, finally,
\begin{align*}
 \mathcal I_2(L,r)&=-A_n\int_{\B_{n-1}} \int_0^{2\pi} \frac {r^{2n-1}}{|1-r^2w_1|^{2L}-(1-r^2)^{2L}}\frac{1-r^2}{1-r^2w_1} w_1 \frac{d\psi}{2\pi}\wedge d\beta_{n-1}(w)\\
 &=-\frac{ r^{2n-1}A_n}{(n-2)!}\int_{\D}\frac {(1-|w|^2)^{n-2}}{|1-r^2w_1|^{2L}-(1-r^2)^{2L}}\frac{1-r^2}{1-r^2w_1} w_1 \frac{dm(w_1)}{\pi}\ .
\end{align*}

\underline{$\mathcal I_3(L,r)$}. Here the form to integrate is, in the terms of the parametrisation,
\begin{align*}
 \Gamma_3&=\sum_{j,k=1}^n \eta_j\eta_k \gamma_j(r\xi)\gamma_k(r\eta)= \sum_{k=1}^n \eta_k \sum_{j=1}^n v_j^1 \left(\sum_{m=1}^n \bar v_j^m \gamma_m(r\alpha)\right)\gamma_k(r\eta)\\
 &=\sum_{k=1}^n \eta_k \left[\sum_{m=1}^n\left(\sum_{j=1}^n v_j^1 \bar v_j^m\right) \gamma_m(r\alpha)\right] \gamma_k(r\eta)
 =r^{2n-1}\left(\sum_{k=1}^n \eta_k \gamma_k(r\eta)\right) \bar w_1 \frac{d\psi}{2\pi}\wedge d\beta_{n-1}(w)\ .
 \end{align*}
Hence
 \begin{align*}
  \mathcal I_3(L,r)&=-A_n\int_{\B_{n-1}} \int_0^{2\pi} \frac {r^{2n-1}}{|1-r^2w_1|^{2L}-(1-r^2)^{2L}}\frac{1-r^2}{1-r^2 \bar w_1}  \bar w_1 \frac{d\psi}{2\pi}\wedge d\beta_{n-1}(w)\\
 &=-\frac{r^{2n-1}A_n}{(n-2)!}\int_{\D}\frac {(1-|w|^2)^{n-2}}{|1-r^2w_1|^{2L}-(1-r^2)^{2L}}\frac{1-r^2}{1-r^2 \bar w_1} \bar w_1 \frac{dm(w_1)}{\pi} \ .
 \end{align*}

Finally, adding up the expressions of $\mathcal I_j(L,r)$ and using that
\begin{align*}
 1+\frac{(1-r^2)^2}{|1-r^2w_1|^{2}}-\frac{1-r^2}{1-r^2w_1} w_1 -\frac{1-r^2}{1-r^2 \bar w_1} \bar w_1 =\left|1-\frac{(1-r^2) w_1}{1-r^2w_1}\right|^2
 =\frac{|1-w_1|^2}{|1-r^2w_1|^{2}},
\end{align*}
we obtain
\[
\mathcal I(L,r)=\sum_{m=1}^4 \mathcal I_m(L,r)=\frac{r^{2n-1} A_n}{(n-2)!}\int_{\D} \frac {(1-|w|^2)^{n-2}}{|1-r^2w_1|^{2L}-(1-r^2)^{2L}} \frac{|1-w_1|^2}{|1-r^2w_1|^{2}}\;  \frac{dm(w_1)}{\pi}\ .
\]
Plugging this into \eqref{reduction} and writing down the value $A_n$ (see \eqref{an}) we get Theorem~\ref{integral}(a).

\bigskip

(b) We just need to repeat the proof of (a) replacing $\beta$ by $\omega$. The corresponding bipotential formula for the variance is now
\begin{multline*}
 \Var  E_{f_L}(r)= 
 \int\limits_{S(0,r)} \int\limits_{S(0,r)} \sum_{j, k=1}^n\frac{\partial^2 \rho_l}{\partial \bar z_j \partial\bar w_k}(z,w)\; 
 \frac i{2\pi} d\bar z_j\wedge \omega_{n-1}(z) \wedge \frac i{2\pi}d\bar w_k \wedge \omega_{n-1}(w) \ .
\end{multline*}
Replacing $\gamma_j(z)$ by
\[
 \Gamma_j(z)=\frac{i}{2\pi} d\bar z_j \wedge\omega_{n-1}(z)
\]
in the calculations above we get the ``invariant'' version of \eqref{reduction}:
\[
\Var I_{f_L}(r)=  L^2 (1-r^2)^{2(L-1)} r^2\sum_{m=1}^4 \mathcal I^I_m(L,r)\ ,
\]
where the integrals $\mathcal I^I_m(L,r)$ are obtained from $\mathcal I_m(L,r)$ replacing the $\gamma_j$ by $\Gamma_j$.

As before, change now $\alpha=\mathcal U^{-1}(\xi)\in S$  and parametrise $\alpha$ with the same coordinates coordinates $w=(w_1,\dots,w_{n-1})\in\B_{n-1}$, $\psi\in[0,2\pi)$. We will be done as soon as we prove the following lemma.

\begin{lemma} For $r<1$ and $j=1,\dots,n$
\[
 \Gamma_j(r\alpha)=\frac{\gamma_j(r\alpha)}{(1-r^2)^{n-1}}\ .
\]
\end{lemma}

\begin{proof}
 From the definition we have
 \begin{multline*}
  \omega_{n-1}(z)=\frac 1{(1-|z|^2)^n} \Bigl\{ \sum_{k=1}^n (1-|z_k|^2)\bigwedge_{j\neq k}\frac i{2\pi} dz_j\wedge d\bar z_j +\Bigr. \\
  \Bigl. +\sum_{\stackrel{j,k=1}{j\neq k}}^n \bar z_j z_k \frac i{2\pi} dz_j\wedge d\bar z_k\wedge\bigwedge_{l\neq j,k} \frac i{2\pi} dz_l \wedge d\bar z_l \Bigr\} \ 
 \end{multline*}
Therefore
\[
 \Gamma_m(z)=\frac i{2\pi}d\bar z_k\wedge\omega_{n-1}(z)= \frac 1{(1-|z|^2)^n}\bigl[(1-|z_m|^2) \gamma_m(z)-\sum_{j\neq m} \bar z_m z_j \gamma_j(z)\bigr]\ ,
\]
and in particular
\[
 \Gamma_m(r\alpha) =\frac 1{(1-r^2)^n}\bigl[(1-r^2|\alpha_m|^2) \gamma_m(r\alpha)-r^2\sum_{j\neq m} \bar \alpha_m \alpha_j \gamma_j(r\alpha)\bigr]\ .
\]
For $m<n$ we have then
\begin{align*}
 \Gamma_m(r\alpha)& =\frac 1{(1-r^2)^n}\bigl[(1-r^2|w_m|^2) \gamma_m(r\alpha)-r^2\sum_{j\neq m} \bar w_m w_j \gamma_j(r\alpha)
 -r^2\sqrt{1-|w|^2} e^{i\psi}\bar w_m \gamma_n(r\alpha)\bigr]\\
 &=\frac {r^{2n-1}}{(1-r^2)^n}\Bigl[(1-r^2|w_m|^2) \bar w_m \frac{d\psi}{2\pi}\wedge d\beta_{n-1}(w)-
 r^2 \sum_{j\neq m} \bar w_m |w_j|^2 \frac{d\psi}{2\pi}\wedge d\beta_{n-1}(w)- \Bigr.\\
 &\qquad\qquad\qquad\qquad\qquad\qquad-r^2 \Bigl. (1-|w|^2)  \bar w_m  \frac{d\psi}{2\pi}\wedge d\beta_{n-1}(w)\Bigr]\\
 &=\frac {r^{2n-1}}{(1-r^2)^n}\bigl[1-r^2|w_m|^2-r^2 \sum_{j\neq m} |w_j|^2-r^2(1-|w|^2) \bigr] \bar w_m  \frac{d\psi}{2\pi}\wedge d\beta_{n-1}(w)\\
 &=\frac {r^{2n-1}}{(1-r^2)^{n-1}}\bar w_m  \frac{d\psi}{2\pi}\wedge d\beta_{n-1}(w) 
\end{align*}
By Lemma~\ref{change} we have thus $\Gamma_m(r\alpha)= (1-r^2)^{1-n}\Gamma_m(r\alpha)$, $m<n$.

If $m=n$ we have similarlly
\begin{align*}
 \Gamma_m(r\alpha)& =\frac {r^{2n-1}}{(1-r^2)^n}\Bigl[(1-r^2(1-|w|^2))\sqrt{1-|w|^2} e^{-i\psi}  \frac{d\psi}{2\pi}\wedge d\beta_{n-1}(w) -\Bigr.\\
&\qquad\qquad\qquad\qquad \Bigl. -r^2\sum_{j=1}^{n-1} w_j \sqrt{1-|w|^2} e^{-i\psi} \bar w_j\frac{d\psi}{2\pi}\wedge d\beta_{n-1}(w)\Bigr]\\
&= \frac {r^{2n-1}}{(1-r^2)^{n-1}}  \sqrt{1-|w|^2} e^{-i\psi}\frac{d\psi}{2\pi}\wedge d\beta_{n-1}(w).
\end{align*}
Again by Lemma~\ref{change} we have $\Gamma_n(r\alpha)= (1-r^2)^{1-n}\Gamma_n(r\alpha)$.
\end{proof}

\section{Proof of the Theorem~\ref{th-L}}

We start with the following consequence of Theorem~\ref{integral}(a). 

\begin{lemma}\label{heart} Let $n\geq 2$. For $L>0$ and $r\in(0,1)$,
\[
 \Var E_{f_L}(r)=\frac{2   L^2  (1-r^2)^{n-2}}{\pi (n-1)! (n-2)!}\  K(L,r)\ ,
\]
where
\[
 K(L,r)=\int_{\frac{1-r^2}{1+r^2}}^1  \frac{s^{2L-n}}{1-s^{2L}} \int_0^{\alpha(s,r)} \left(2\cos\theta-2\cos \alpha(s,r)\right)^{n-2} \bigl(s+\frac 1s-2\cos\theta\bigr)  d\theta\; ds
\]
and $\alpha(s,r)=\arccos\bigl[\dfrac 12\bigl((1+r^2)s+\dfrac{1-r^2}{s}\bigr)\bigr]$.
\end{lemma}

\begin{proof}
 
Starting from the expression given in Theorem~\ref{integral}(a), write
\begin{equation}\label{integral-J}
 \Var E_{f_L}(r)=\frac{r^{4n} L^2 (1-r^2)^{-4}}{\pi (n-1)! (n-2)!} J(L,r)\ ,
\end{equation}
where
\begin{equation*}
 J(L,r)=\int_{\D}
 \frac {(1-|w|^2)^{n-2}}{1-\left(\frac{1-r^2}{|1-r^2w|}\right)^{2L}} \left(\frac{1-r^2}{|1-r^2 w|}\right)^{2L+2}\; |1-w|^2dm(w)
\end{equation*}

We write this integral in (a sort of) polar coordinates $s$, $\theta$: let
\[
 w=\frac 1{r^2}+\frac{1-r^2}{r^2 s} e^{i(\pi-\theta)}\ ,
\]
so that 
\[
 \frac{1-r^2}{|1-r^2 w|}=s\ .
\]
Let $p(s,r)$ denote the intersection point of the unit circle and the circle $|w-1/r^2|=\frac{1-r^2}{r^2 s}$. 
In these coordinates, $w\in\D$ if and only if $s\in (\frac{1-r^2}{1+r^2},1)$ and $\theta\in (-\alpha(s,r),\alpha(s,r))$, where $\alpha(s,r)$ is the angle between the complex numbers $-1/r^2$ and $p(s,r)-1/r^2$.

In these coordinates
\begin{equation}\label{dist}
 1-|w|^2= 1-\left|\frac 1{r^2}+\frac{1-r^2}{r^2 s} e^{i(\pi-\theta)}\right|^2=\frac{1-r^2}{r^4 s}\left[2 \cos\theta-2\cos\alpha(s,r)\right]\ ,
\end{equation}
since
\begin{equation}\label{alpha}
 2\cos\alpha(s,r)=(1+r^2)s+\frac{1-r^2}{s}\ .
\end{equation}
Also
\begin{align*}
  |1-w|^2&= \left|1-\frac 1{r^2}+\frac{1-r^2}{r^2 s} e^{i(\pi-\theta)}\right|^2=
  \frac{(1-r^2)^2}{r^4s}\left(s+\frac 1s-2 \cos\theta\right)\ .
\end{align*}

Since $dm(w)=\dfrac{(1-r^2)^2}{r^4 s^3}\ d\theta\; ds$, and letting $\D_+=\{z\in \D : \Im z >0\}$,  
\begin{align*}
 J(L,r)&=2\int_{\D^+}\frac {(1-|w|^2)^{n-2}}{1-\left(\frac{1-r^2}{|1-r^2w|}\right)^{2L}} \left(\frac{1-r^2}{|1-r^2w|}\right)^{2L+2}\; |1-w|^2dm(w)
 =\frac{2(1-r^2)^{n+2}}{r^{4n} } K(L,r)\ ,
\end{align*}
where
\[
 K(L,r)=\int_{\frac{1-r^2}{1+r^2}}^1 \int_0^{\alpha(s,r)} \frac{s^{2L-n}}{1-s^{2L}}\left(2\cos\theta-2\cos\alpha(s,r)\right)^{n-2} \left(s+\frac 1s-2\cos\theta\right)  d\theta\; ds\ .
\]
Going back to \eqref{integral-J} we obtain Lemma~\ref{heart}.
\end{proof}

Our starting point in the proof of Theorem~\ref{th-L} is the expression of $\Var E_{f_L}(r)$ given by Lemma~\ref{heart}. 

Fix $r<1$ and, in order to simplify the notation, let $\alpha(s)=\alpha(s,r)$.

1. In the first place we observe that for the asymptotics of $\Var E_{f_L}(r)$ as $L\to\infty$, only the part of the integral $K(L,r)$ corresponding to $s$ 
close to 1 is relevant. Fix $\epsilon>0$ and let us see that the part of the integral up to $s=1-\epsilon$ decays exponentially in $L$; we will see later that the part corresponding to $s\in(1-\epsilon,1)$ decays polynomially. From \eqref{dist} we have
\[
 2 \cos\theta-2\cos\alpha(s,r)\leq \frac{r^4 s}{1-r^2}
\]
and therefore, for $L$ big enough,
\begin{multline*}
 \int_{\frac{1-r^2}{1+r^2}}^{1-\epsilon}  \frac{s^{2L-n}}{1-s^{2L}} \int_0^{\alpha(s)}\left(2\cos\theta-2\cos\alpha(s,r)\right)^{n-2} \left(s+\frac 1s-2\cos\theta\right)  d\theta\; ds\\
\leq \frac{(1-\epsilon)^{2L-n}}{1-(1-\epsilon)^{2L}} \int_{\frac{1-r^2}{1+r^2}}^{1-\epsilon} \left(\frac{r^4 s}{1-r^2}\right)^{n-2} 
 \left[\left(s+\frac 1s\right)\alpha(s)-2\sin\alpha(s)\right]  d\theta\; ds\\
 \leq \frac{(1-\epsilon)^{2L-2} r^{4n-2}}{(1-r^2)^{n-2}} \int_{\frac{1-r^2}{1+r^2}}^{1} \left[\left(s+\frac 1s-2\right)\alpha(s)+2\alpha(s)\right]\; ds\\
 \leq \frac{(1-\epsilon)^{2L-2} r^{4n-2}}{(1-r^2)^{n-2}} \int_{\frac{1-r^2}{1+r^2}}^{1} \left[\left(s+\frac 1s-2\right)\frac{\pi}2 +\pi\right]\; ds =C_{n,r} (1-\epsilon)^{2L-2}\ .
 \end{multline*}

2. For $s$ near 1 we can assume that $\alpha(s)$ is small. More precisely, given $\delta>0$, there exists $\epsilon=\epsilon(\delta)>0$ such that $\lim\limits_{\delta\to 0}\epsilon(\delta)=0$ and
\[
 s\in (1-\epsilon,1)\quad\Longrightarrow\quad \alpha(s)\leq\arccos(1-\delta)=O(\sqrt\delta)\ .
\]
To see this notice that, by \eqref{alpha}, $\alpha(s)\leq\arccos(1-\delta)$ precisely when
$(1+r^2)s+\frac{1-r^2}s\geq  2 (1-\delta)$, which is equivalent to
$s^2-\frac{2(1-\delta)}{1+r^2}\; s+\frac{1-r^2}{1+r^2}\geq 0$, and to
\[
 s\notin\left(\frac 1{(1+r^2)}(1-\delta-\sqrt{(1-\delta)^2-(1-r^4)}), \frac 1{(1+r^2)}(1-\delta+\sqrt{(1-\delta)^2-(1-r^4)})\right)\ .
\]
Hence it is enough to take
\[
 \epsilon=1-\frac 1{(1+r^2)}(1-\delta+\sqrt{(1-\delta)^2-(1-r^4)})\ .
\]

3. For $\alpha(s)$ close to 0,
 \[
  \alpha^2(s)=(1+r^2)\; \frac{1-s}s\;  \left(s-\frac{1-r^2}{1+r^2}\right)+\cdots
 \]
(here and in the remaining of this section, $+\cdots$ will indicate terms of lower order as $s\to 1^{-}$.)

 To see this notice that, from \eqref{alpha},

 \begin{align}\label{2-2cos}
  2-2\cos\alpha(s)&=-\frac{(1-s)^2}s+\frac{r^2}s-sr^2
  =(1+r^2)\frac{1-s}s \left(s-\frac{1-r^2}{1+r^2}\right)\ ,
 \end{align}
 which together with the approximation $  2-2\cos\alpha(s)=\alpha^2(s)+o(\alpha^2(s))$ gives the statement.

4. We will see next that the whole integral, with $s$ from $\frac{1-r^2}{1+r^2}$ up to 1, has polynomial decay.
As seen in the previous point, for $s$ close to 1 we have $\alpha^2(s)=2r^2  (1-s)+\cdots$
On the other hand, by \eqref{alpha}
\begin{align*}
 2\cos\theta -2\cos\alpha(s)&=-2(1-\cos\theta)+2-s-\frac 1s+r^2(\frac 1s-s)\\
 &=2r^2(1-s) -2(1-\cos\theta)-(1-r^2)\sum_{j=2}^\infty (1-s)^j\ ,
\end{align*}
and
\begin{align*}
 s+\frac 1s-\cos\theta=s+\frac 1s-2+2(1-\cos\theta)=2(1-\cos\theta)+\sum_{j=2}^\infty (1-s)^j\ .
\end{align*}

Together with the approximation $2(1-\cos\theta)=\theta^2+\cdots$, and writing only the leading term around $s=1$, this yields
\begin{align*}
 A(s):&=\int_0^{\alpha(s)}\left(2\cos\theta-2\cos\alpha(s)\right)^{n-2} \left(s+\frac 1s-2\cos\theta\right)  d\theta\\
 &=\int_{0}^{\sqrt {2r^2 (1-t)}} \bigl(2r^2(1-s) -\theta^2\bigr)^{n-2}\; \theta^2\; d\theta+\cdots\\
 &=\sum_{j=0}^{n-2}\binom{n-2}{j}\bigl(2 r^2(1-s)\bigr)^{n-2-j} (-1)^j\int_{0}^{\sqrt {2 r^2(1-t)}} \theta^{2j+2} d\theta+\cdots\\
 &=\bigl(2 r^2(1-s)\bigr)^{n-1/2} \sum_{j=0}^{n-2} \binom{n-2}{j} \frac{(-1)^j}{2j+3}+\cdots
\end{align*}

A straightforward computation shows that for $m\in\N$ and $z\in\C\setminus\{0,1,\dots,m\}$,
\begin{equation}\label{Cm} 
\sum_{j=0}^{m} \binom{m}{j} \frac{(-1)^j}{z+j}=\frac{m!}{z(z+1)\cdots(z+m)}=\frac{m!\Gamma(z)}{\Gamma(z+m+1)}\ .
\end{equation}

Applying this to $z=3/2$, $m=n-2$, and using that $\Gamma(1/2)=\sqrt{\pi}$ we have
\[
 A(s)= \frac{(n-2)!\sqrt{\pi}}{4\Gamma(n+1/2)}\bigl(2 r^2(1-s)\bigr)^{n-1/2}\ .
\]

Since for $L$ big the integral for $s\in(0,\frac{1-r^2}{1+r^2})$ tends to 0, this implies that
\begin{align*}
 K(L,r)&=\frac{(n-2)!\sqrt{\pi}}{4\Gamma(n+1/2)}\int_{\frac{1-r^2}{1+r^2}}^{1}  \frac{s^{2L-n}}{1-s^{2L}}  \bigl(2 r^2(1-s)\bigr)^{n-1/2} ds+\cdots\\
 &=\sqrt{\frac{\pi}2} \frac{(n-2)! 2^{n-2}}{\Gamma(n+1/2)} r^{2n-1}\sum_{k=0}^\infty\int_{0}^{1} s^{2L+2LK-n} (1-s)^{n-1/2}ds +\cdots\\
&=\sqrt{\frac{\pi}2} \frac{(n-2)! 2^{n-2}}{\Gamma(n+1/2)}  r^{2n-1}\sum_{k=0}^\infty \frac{\Gamma(2L+2Lk-n+1)\Gamma(n+1/2)}{\Gamma(2L+2Lk+3/2)}+\cdots
\end{align*}
Using the asymptotics $\lim\limits_{k\to\infty}\frac{\Gamma(k+a)}{\Gamma(k) k^a}=1$  we have then
\begin{align*}
 K(L,r)&=\sqrt{\frac{\pi}2} \frac{(n-2)! 2^{n-2}}{\Gamma(n+1/2)}  r^{2n-1}\Gamma(n+1/2) \sum_{k=0}^\infty\frac 1{(2L+2kL)^{n+1/2}}+\cdots\\
  &=\frac{\sqrt{\pi}}8 (n-2)!\zeta(n+1/2) r^{2n-1} \frac 1{L^{n+1/2}}+\cdots
 \end{align*} 
Finally, by Lemma~\ref{heart} we get 
\begin{align*}
 \Var E_{f_L}(r)&=\frac{2L^2 (1-r^2)^{n-2}}{\pi (n-1)! (n-2)!}  \frac{\sqrt{\pi}}8 (n-2)! \zeta\bigl(n+  1/2\bigr) r^{2n-1} \frac 1{L^{n+1/2}}+\cdots\\
 &=\frac 1{4\sqrt{\pi}}\frac{\zeta(n+1/2)}{(n-1)!} r^{2n-1} (1-r^2)^{n-2} L^{3/2-n}+\cdots
\end{align*}
This and Theorem~\ref{integral}(b) finish the proof.

\section{Proof of the Theorem~\ref{th-r}}

Let us see first that the order of growth as $r\to 1^{-}$ is as stated. Later on we will see how the constants $C(L,n)$ can be determined.
The notation $\simeq$ indicates that there exists a constant $C$ independent of $r$ such that $C^{-1}A\leq B\leq CA$. 

According to Lemma~\ref{heart} it is enough to study the asymptotics of $K(L,r)$ as $r\to 1^{-}$. 

Since $ 2\cos x-2\cos a \simeq (\sin a) (a-x)$ for $0\leq x\leq a\leq \pi/2$, 
we see that in the range of integration of $\theta$ in $K(L,r)$
\[
 2\cos\theta-2\cos\alpha(s,r)\simeq (\sin \alpha(s,r)) (\alpha(s,r)-\theta)\simeq \alpha(s,r) (\alpha(s,r)-\theta)\ .
\]
Therefore
\[
 K(L,r)\simeq \int_{\frac{1-r^2}{1+r^2}}^1 \frac{s^{2L-n}}{1-s^{2L}} (\alpha(s,r))^{n-2}\int_0^{\alpha(s,r)} (\alpha(s,r)-\theta)^{n-2}
 \left[\frac{(1-s)^2}{s}+2(1-\cos\theta)\right]d\theta\; ds
\]
Denote temporarily $\alpha=\alpha(s,r)$. Using now that $1-\cos\theta\simeq \theta^2$ for $\theta\in[0,\pi/2]$ 
we can estimate the integral in $\theta$:
\begin{multline*}
 \int_0^\alpha (\alpha-\theta)^{n-2}\left[\frac{(1-s)^2}{s}+2(1-\cos\theta)\right]d\theta \simeq\\
 \simeq\frac{(1-s)^2}{s}\int_0^\alpha (\alpha-\theta)^{n-2} d\theta +\int_0^\alpha (\alpha-\theta)^{n-2} \theta^2 d\theta\\
 \simeq \frac{(1-s)^2}{s}\int_0^\alpha (\alpha-\theta)^{n-2} d\theta+ \int_0^\alpha (\alpha-\theta)^{n} d\theta+2\alpha\int_0^\alpha (\alpha-\theta)^{n-1} d\theta+\alpha^2 \int_0^\alpha (\alpha-\theta)^{n-2} d\theta\\
 \simeq \frac{(1-s)^2}{s} \alpha^{n-1} + \alpha^{n+1}=\alpha^{n-1}\left[\frac{(1-s)^2}{s}+\alpha^2\right]\ .
\end{multline*}
Hence
\[
 K(L,r)\simeq \int_{\frac{1-r^2}{1+r^2}}^1 \frac{s^{2L-n}}{1-s^{2L}} (\alpha(s,r))^{2n-3}\left[\frac{(1-s)^2}{s}+\alpha^2(s,r)\right]\; ds\ .
\]
Using \eqref{2-2cos} we have also,
\[
 \alpha(s,r)^2\simeq 2-2\cos\alpha(s,r)=(1+r^2)\frac{1-s}{s}\left(s-\frac{1-r^2}{1+r^2}\right)\ ,
\]
and therefore
\begin{align*}
 K(L,r)&\simeq\int_{\frac{1-r^2}{1+r^2}}^1 \frac{s^{2L-n}}{1-s^{2L}}\left(\frac{1-s}{s}\bigl(s-\frac{1-r^2}{1+r^2}\bigr)\right)^{n-3/2}\left[\frac{(1-s)^2}{s}+\frac{1-s}s\bigl(s-\frac{1-r^2}{1+r^2}\bigr)\right]\; ds\\
 &\simeq \int_{\frac{1-r^2}{1+r^2}}^1 \frac{s^{2L-2n+1/2}}{1-s^{2L}} (1-s)^{n-1/2} \bigl(s-\frac{1-r^2}{1+r^2}\bigr)^{n-3/2}\bigl((1-s)+(s-\frac{1-r^2}{1+r^2})\bigr) ds \\
 &\simeq \int_{\frac{1-r^2}{1+r^2}}^1 \frac{s^{2L-2n+1/2}}{1-s^{2L}} (1-s)^{n-1/2} \bigl(s-\frac{1-r^2}{1+r^2}\bigr)^{n-3/2} ds\ .
\end{align*}
Let us see now that for the asymptotics as $r\to 1^{-}$ it is enough to take care of $s$ small. Notice, for instance, that the portion of the integral where $s\in(1/2,1)$ tends to the constant
\[
 \int_{1/2}^1 \frac{s^{2L-2n+1/2}}{1-s^{2L}} (1-s)^{n-1/2} s^{n-3/2} ds=\int_{1/2}^1 \frac{s^{2L-n-1}}{1-s^{2L}} (1-s)^{n-1/2}  ds\ .
\]
On the other hand, for $s\leq 1/2$ we have $1-s\simeq 1-s^{2L}\simeq 1$ and therefore,
\begin{multline*}
 \int_{\frac{1-r^2}{1+r^2}}^{1/2} \frac{s^{2L-2n+1/2}}{1-s^{2L}} (1-s)^{n-1/2} \bigl(s-\frac{1-r^2}{1+r^2}\bigr)^{n-3/2} ds\simeq
  \int_{\frac{1-r^2}{1+r^2}}^{1/2} s^{2L-2n+1/2}\bigl(s-\frac{1-r^2}{1+r^2}\bigr)^{n-3/2} ds\\
  \simeq
 \begin{cases}
  (1-r^2)^{2L-n}+\cdots\; &\textrm{if $2L-n<0$}\\
  \log\dfrac 1{1-r^2}+\cdots \; &\textrm{if $2L-n=0$}\\
  \quad 1\; &\textrm{if $2L-n>0$}.
 \end{cases}
 \qquad\qquad
\end{multline*}

\medskip

This gives the order of growth of $\Var E_{f_L}(r)$ as stated in Theorem~\ref{th-r}. Once we know this we can determine the values $C(L,n)$, $L>0$.

\underline{Case $L\leq n/2$}. Here $K(L,r)$ tends to $\infty$ as $r\to 1^{-}$, at speed $(1-r^2)^{2L-n}$, so it is enough to consider the terms giving  this maximal order of growth. Since
\[
 s+\frac 1s-2\cos\theta=\frac{(1-s)^2}s+2(1-\cos\theta)=\frac 1s+\cdots
\]
and
\[
\dfrac 1{1-s^{2L}}=1+\cdots 
\]
(where the dots indicate lower order terms) we have 
\begin{align*}
 K(L,r)&=\int_{\frac{1-r^2}{1+r^2}}^{1} s^{2L-n-1}\int_0^{\alpha(s,r)} (2\cos\theta-2\cos\alpha(s,r))^{n-2}d\theta\; ds+\cdots\\
 &=\int_{\frac{1-r^2}{1+r^2}}^{1} \int_0^{\alpha(s,r)} s^{2L-n-1} \sum_{j=0}^{n-2} \binom{n-2} {j} (-1)^j (2\cos\alpha(s,r))^j (2\cos\theta)^{n-2-j}d\theta\; ds+\cdots
\end{align*}

In order to apply Fubini's theorem --and to simplify the notation-- denote $\epsilon(r)=\frac{1-r^2}{1+r^2}$. The domain of the double integral above is thus given by the conditions
\[
 \begin{cases}
  \epsilon(r)\leq s\leq 1 \\
  0\leq\theta\leq \alpha(s,r)\ .
 \end{cases}
\]
To determine the global range of $\theta$, notice that the function
\[
 h(s):=2\cos \alpha(s,r)=(1+r^2)s+\frac{1-r^2}s
\]
has a minimum at $s=\sqrt{\epsilon(r)}$ and that $h(\sqrt{\epsilon(r)})=2\sqrt{1-r^4}$. Therefore
\[
 0\leq\theta\leq\alpha_0(r):=\arccos(\sqrt{1-r^2})\ .
\]
Once $\theta$ is fixed, the inequalities above determine the range of $s$:
\[
s\in A(\theta):=\{s : \epsilon(r)\leq s\leq 1\; ,\; 2\cos\theta\geq (1+r^2)s+\frac{1-r^2}s\}\ .
\]
Therefore
\begin{equation}\label{preJ}
 K(L,r)=\sum_{j=0}^{n-2} \binom{n-2}{j} (-1)^j \int_0^{\alpha_0(r)} (2\cos\theta)^{n-2-j}
 \left[\int_{A(\theta)}(2\cos\alpha(s,r))^j ds\right]\, d\theta+\cdots
\end{equation}
Here we wish to determine the leading term of
\[
 J:=\int_{A(\theta)}(2\cos\alpha(s,r))^j ds=\int_{A(\theta)} \bigl((1+r^2)s+\frac{1-r^2}s\bigr)^j ds\ .
\]
After the change of variable $s=\sqrt{\epsilon(r)} x$ the restrictions on $s$ imposed by $A(\theta)$ are equivalent to the conditions
\[
\sqrt{\epsilon(r)} \leq x\leq 1/\sqrt{\epsilon(r)} \quad , \quad (x+1/x)\sqrt{1-r^4}\leq 2\cos\theta\ ,
\]
hence
\begin{align}\label{J}
 J&=(\epsilon(r))^{L-\frac n2} (1-r^4)^{j/2}
 \int\limits_{\stackrel{\sqrt{\epsilon(r)} \leq x\leq 1/\sqrt{\epsilon(r)}}{x+1/x\leq \frac{2\cos\theta}{\sqrt{1-r^4}}}}
 x^{2L-n-1}(x+\frac 1x)^j dx+\cdots \\
 &=\frac{(1-r^2)^{L-\frac n2+\frac j2}}{2^{L-\frac n2-\frac j2}}
  \int\limits_{\stackrel{\sqrt{\epsilon(r)} \leq x\leq 1/\sqrt{\epsilon(r)}}{x+1/x\leq \frac{2\cos\theta}{\sqrt{1-r^4}}}}
 x^{2L-n-1}(x+\frac 1x)^j dx+\cdots\nonumber
\end{align}
We split the integral above into two parts, depending on whether or not $x\leq 1$. For both parts we perform the same change of variables $y=x+1/x$ and extract the terms of maximal order.

(i) Assume $x\leq 1$. Here $y=1/x+\cdots$ and therefore $x^{2L-n-1}=y^{n+1-2L}+\cdots$. Similarly 
\[
dy=(1-\frac 1{x^2})dx=-\frac{dx}{x^2}+\cdots=-y^2 dx+\cdots
\]
As for the limits of integration, notice that if $x=\sqrt{\epsilon(r)} $ then
$y=\sqrt{\epsilon(r)} +1/\sqrt{\epsilon(r)} =\frac 2{\sqrt{1-r^4}}$. This yields
\begin{align}\label{A1}
 A_1:&=\int\limits_{\stackrel{\sqrt{\epsilon(r)} \leq x\leq 1}{x+1/x\leq \frac{2\cos\theta}{\sqrt{1-r^4}}}}
 x^{2L-n-1}(x+\frac 1x)^j dx
 =\int_2^{\frac{2\cos\theta}{\sqrt{1-r^4}}} y^{n+1-2L} y^j\frac{dy}{y^2}+\dots\\
 &=\frac{1}{n-2L+j}\left(\frac{2\cos\theta}{\sqrt{1-r^4}}\right)^{n-2L+j}+\cdots
 =\frac{2^{\frac n2-L+\frac j2} (\cos\theta)^{n-2L+j}}{(n-2L+j) (1-r^2)^{\frac n2-L+\frac j2}}+\cdots \nonumber
\end{align}

(ii) Let us see now that the integral corresponding to $x\geq 1$ is of smaller order. Here $y=x+\cdots$ and $dy=dx+\cdots$, thus
\begin{align*}
 A_2:&=\int\limits_{\stackrel{1 \leq x\leq 1/\sqrt{\epsilon(r)}}{x+1/x\leq \frac{2\cos\theta}{\sqrt{1-r^4}}}}
 x^{2L-n-1}(x+\frac 1x)^j dx
 =\int_2^{\frac{2\cos\theta}{\sqrt{1-r^4}}} y^{2L-n-1+j} dy+\cdots
\end{align*}
If $2L-n+j<0$ these integrals tend to a constant. If $2L-n+j=0$ then this is O$(\log(\frac 1{1-r^2}))$, which is obviously of smaller order than 
$(1-r^2)^{L-\frac n2-\frac j2}=(1-r^2)^{-j}$. Finally, if $2L-n+j>0$ then
\[
 A_2=\frac{1}{2L-n+j} \left(\frac{2\cos\theta}{\sqrt{1-r^4}}\right)^{2L-n+j}= o\bigl(\frac 1{(1-r^2)^{\frac n2-L+\frac j2}}\bigr)\ .
\]

Therefore the leading term of the integral in \eqref{J} is $A_1$, and by \eqref{A1} we get
\begin{align*}
 J&=\frac{(1-r^2)^{L-\frac n2+\frac j2}}{2^{L-\frac n2-\frac j2}} \frac{2^{\frac n2-L+\frac j2} (\cos\theta)^{n-2L+j}}{(n-2L+j) (1-r^2)^{\frac n2-L+\frac j2}}+\cdots\\
 &=\frac{2^{n-2L+j}}{n-2L+j} (\cos\theta)^{n-2L+j} (1-r^2)^{2L-n}+\cdots
\end{align*}

Plugging this into \eqref{preJ}  we obtain
\begin{align*}
 K(L,r)&=2^{2n-2L-2}\left[\sum_{j=0}^{n-2} \binom{n-2}{j}\frac{(-1)^j}{n-2L+j}\right]\left[\int_0^{\pi/2} (\cos\theta)^{2n-2L-2} d\theta\right] (1-r^2)^{2L-n}+\cdots 
\end{align*}

The sum in $j$ is computed using \eqref{Cm} with $z=n-2L$ and $m=n-2$. The integrals in $\theta$ are taken care of by the following identity, which is a simple computation:
for $m\in\mathbb N$
\begin{equation}\label{cosk}
 \int_0^{\pi/2}(\cos\theta)^m d\theta=\frac{\sqrt\pi}2\frac{\Gamma(\frac{m+1}2)}{\Gamma(\frac m2+1)}\ .
\end{equation}

We have thus 
\[
 K(L,r)=2^{2n-2L-2}\frac{(n-2)!\Gamma(n-2L)}{\Gamma(2n-2L-1)} \frac{\sqrt{\pi}}2 \frac{\Gamma(n-L-1/2)}{\Gamma(n-L)}(1-r^2)^{2L-n}+\cdots
\]
and therefore
\[
 \Var E_{f_L}(r)=\frac{L^2}{\sqrt{\pi}}\frac{2^{2n-2L-2}}{(n-1)!} \frac{\Gamma(n-L-1/2)}{\Gamma(2n-2L-1)}\frac{\Gamma(n-2L)}{\Gamma(n-L)}(1-r^2)^{2L-n}+\cdots
\]
This expression can be simplified by means of the duplication formula for the $\Gamma$-function:
\begin{equation}\label{duplication}
\Gamma(2z)= \frac{2^{2z-1}}{\sqrt{\pi}}\Gamma(z)\Gamma(z+1)\ .
\end{equation}
Taking $z=n-L-1/2$ we see that
\[
 \frac{2^{2n-2L-2}}{\sqrt{\pi}}\frac{\Gamma(n-L-1/2)}{\Gamma(2n-2L-1)}=\frac 1{\Gamma(n-L)}
\]
and therefore
 \[
 \Var E_{f_L}(r)=\frac{L^2}{(n-1)!} \frac{\Gamma(n-2L)}{(\Gamma(n-L)^2}(1-r^2)^{2L-n}+\cdots
\]
Applying once more the duplication formula, now with $z=n/2-L$, we finally get
 \[
 \Var E_{f_L}(r)=\frac{L^2}{(n-1)!} \frac{2^{n-2L-1}}{\sqrt{\pi}}\frac{\Gamma(\frac n2-L) \Gamma(\frac{n+1}2-L) }{(\Gamma(n-L))^2}(1-r^2)^{2L-n}+\cdots
\]

\underline{Case $L=n/2$}. Proceeding as in the previous case we arrive at
\[
 K(L,r)=\sum_{j=0}^{n-2} \ \binom{n-2}{j} (-1)^j \int_0^{\alpha_0(r)} (2\cos \theta)^{n-2-j}\left[\int_{A(\theta)}\frac 1s (2\cos\alpha(s,r))^j ds\right]\; d\theta+\cdots.
\]
where
\[
 J:=\int_{A(\theta)}\frac 1s (2\cos\alpha(s,r))^j ds=2^{j/2}(1-r^2)^{j/2} \int\limits_{\stackrel{\sqrt{\epsilon(r)} \leq x\leq 1/\sqrt{\epsilon(r)}}{x+1/x\leq \frac{2\cos\theta}{\sqrt{1-r^4}}}}
 \frac 1x(x+\frac 1x)^j dx+\cdots
\]
As before we split the integral into two parts, depending on whether $x\leq 1$ or not. Unlike the previous case here both parts have the same order.

(i) Assume $x\leq 1$. Then
\begin{align*}\label{A1}
 A_1:&=\int\limits_{\stackrel{\sqrt{\epsilon(r)} \leq x\leq 1}{x+1/x\leq \frac{2\cos\theta}{\sqrt{1-r^4}}}}
\frac 1x(x+\frac 1x)^j dx =\int_2^{\frac{2\cos\theta}{\sqrt{1-r^4}}} y^{j-1} dy +\cdots\\
 &=
 \begin{cases}
    \log\bigl(\dfrac{2\cos\theta}{\sqrt{1-r^4}}\bigr)+\cdots =\dfrac 12\log\bigl(\dfrac 1{1-r^2}\bigr)+\log(\cos\theta)\cdots \quad &\textrm{if $j=0$} \\
    \dfrac 1j\left(\dfrac{2\cos\theta}{\sqrt{1-r^4}}\right)^j+\cdots=\dfrac {2^{j/2}(\cos\theta)^j}{j(1-r^2)^{j/2}}+\cdots \quad &\textrm{if $j\geq 1$}\ .
   \end{cases}
\end{align*}

(ii) For $x>1$ we have the same values as in the previous case:
\[
 A_2:=\int\limits_{\stackrel{1 < x\leq 1/\sqrt{\epsilon(r)}}{x+1/x\leq \frac{2\cos\theta}{\sqrt{1-r^4}}}}
\frac 1x(x+\frac 1x)^j dx =\int_2^{\frac{2\cos\theta}{\sqrt{1-r^4}}} y^{j-1} dy+\cdots 
\]

Then
\begin{align*}
 J&=
 \begin{cases}
  \log\bigl(\dfrac{1}{1-r^2}\bigr)+2\log(\cos\theta))+\cdots &\textrm{if $j=0$} \\
  \dfrac{2^{j+1}}{j} (\cos\theta)^j+\cdots \quad &\textrm{if $j\geq 1$}\ .
 \end{cases}
 \end{align*}

 With this we get
 \begin{multline*}
  K(L,r)=\int_0^{\alpha_0(r)} (2\cos\theta)^{n-2}\left(  \log\bigl(\dfrac{1}{1-r^2}\bigr)+2\log(\cos\theta))\right)\; d\theta+\\
  +\sum_{j=1}^{n-2} \ \binom{n-2}{j} (-1)^j  \int_0^{\alpha_0(r)} (2\cos\theta)^{n-2-j} \dfrac{2^{j+1}}{j} (\cos\theta)^j d\theta\ .
 \end{multline*}

 Since $ (\cos\theta)^{n-2}\log(\cos\theta)$ is integrable in $[0,\pi/2]$ and the sum in $j\geq 1$ is bounded independently of $r$, the leading term of $K(L,r)$ is given by the factor $\log(1/(1-r^2))$ above. Then \eqref{cosk} yields
 \begin{align*}
  K(L,r)&= 2^{n-2} \int_0^{\alpha_0(r)} (\cos\theta)^{n-2}\log\bigl(\dfrac{1}{1-r^2}\bigr)+\cdots
  = 2^{n-2} \int_0^{\pi/2} (\cos\theta)^{n-2}\log\bigl(\dfrac{1}{1-r^2}\bigr)+\cdots \\
  &=2^{n-3}\sqrt{\pi}\; \frac{\Gamma(\frac{n-1}2)}{\Gamma(\frac n2)} \log\bigl(\dfrac{1}{1-r^2}\bigr)+\cdots
 \end{align*}
By Lemma~\ref{heart} we get then
\[
 \Var E_{f_L}(r)=\frac{2^{n-2}(\frac n2)^2}{\sqrt{\pi} (n-1)! (n-2)!} \frac{\Gamma(\frac{n-1}2)}{\Gamma(\frac n2)} (1-r^2)^{n-2} \log\bigl(\dfrac{1}{1-r^2}\bigr)+\cdots
\]
The duplication formula \eqref{duplication} with $z=\frac{n-1}2$ yields
$\frac{2^{n-2}\Gamma(\frac{n-1}2)}{\sqrt{\pi} (n-2)!}=\frac 1{\Gamma(\frac n2)}$
and thus the stated result.

\underline{Case $L>n/2$}. Here $K(L,r)$ tends, as $r\to 1^{-}$, to the constant
\[
 K(L,1)=\int_0^1\frac{s^{2L-n}}{1-s^{2L}}\int_0^{\arccos s} (2\cos\theta-2s)^{n-2} (s+\frac 1s-2\cos\theta)d\theta\; ds\ .
\]

As before, we expand the power and apply Fubini's theorem:
\begin{align*}
 K(L,1)&=2^{n-2}\int_0^1\int_0^{\arccos s} \sum_{k=0}^\infty s^{2L+2Lk-n} \sum_{j=0}^{n-2} \ \binom{n-2}{j} (-1)^j s^j (\cos\theta)^{n-2-j} (s+\frac 1s-2\cos\theta)d\theta\; ds\\
 &= 2^{n-2} \sum_{k=0}^\infty \sum_{j=0}^{n-2} \ \binom{n-2}{j} (-1)^j\int_0^{\pi/2} (\cos\theta)^{n-2-j}\int_0^{\cos\theta} s^{2L(1+k)-n+j}(s+\frac 1s-2\cos\theta)\; dsd\theta.
\end{align*}
Re-indexing the sum in $k$ and performing the integral in $s$ we get
\begin{multline*}
 K(L,1)=2^{n-2} \sum_{k=1}^\infty \sum_{j=0}^{n-2} \ \binom{n-2}{j} (-1)^j\left[ \int_0^{\pi/2}\frac{(\cos\theta)^{2Lk}}{2Lk-n+j+2}\; d\theta +
 \int_0^{\pi/2}\frac{(\cos\theta)^{2Lk-2}}{2Lk-n+j}\; d\theta \right. \\
 \qquad\qquad\qquad\qquad \left. -2\int_0^{\pi/2}\frac{(\cos\theta)^{2Lk}}{2Lk-n+j+1}\; d\theta \right]
\end{multline*}
Using \eqref{cosk} to compute the integrals in $\theta$, and \eqref{Cm} to compute the sum in $j$, we obtain 
\begin{align*}
 K(L,1)&=\sqrt{\pi}(n-2)!2^{n-3} \sum_{k=1}^\infty \left[ \frac{\Gamma(Lk+1/2)}{\Gamma(Lk+1)} \frac{\Gamma(2Lk-n+2)}{\Gamma(2Lk+1)} +
 \frac{\Gamma(Lk-1/2)}{\Gamma(Lk)} \frac{\Gamma(2Lk-n)}{\Gamma(2Lk-1)} \right. \\
& \left. -2 \frac{\Gamma(Lk+1/2)}{\Gamma(Lk+1)} \frac{\Gamma(2Lk-n+)}{\Gamma(2Lk)}\right]\ .
\end{align*}

\begin{lemma}\label{S} For $M>0$
\begin{multline*}
 \frac{\Gamma(M+1/2)}{\Gamma(M+1)} \frac{\Gamma(2M-n+2)}{\Gamma(2M+1)} +
 \frac{\Gamma(M-1/2)}{\Gamma(M)} \frac{\Gamma(2M-n)}{\Gamma(2M-1)}  -2 \frac{\Gamma(M+1/2)}{\Gamma(M+1)} \frac{\Gamma(2M-n+1)}{\Gamma(2M)}\\
 =2^{-n} \frac{\Gamma(M-\frac n2) \Gamma(M-\frac{n-1}2)}{(\Gamma(M+1))^2} (M+\frac{n(n-1)}2)
\end{multline*}
\end{lemma}

This with $M=Lk$ yields 
\[
 K(L,1)=\frac{\sqrt{\pi}(n-2)!}{8} \sum_{k=1}^\infty \frac{\Gamma(Lk-n/2) \Gamma(Lk-(n-1)/2)  }{(\Gamma(Lk+1))^2} (LK+\frac{n(n-1)}2),
\]
which together with the identity of Lemma~\ref{heart} gives the stated result.

\begin{proof}[Proof of Lemma~\ref{S}] Denote by $S$ the left hand side of identity in the lemma. Then
\begin{align*}
 S&= \frac{\Gamma(M+1/2)}{\Gamma(M+1)} \frac{\Gamma(2M-n+2)}{\Gamma(2M+1)} - \frac{\Gamma(M+1/2)}{\Gamma(M+1)} \frac{\Gamma(2M-n+)}{\Gamma(2M)}+\\
&\qquad\qquad\qquad\qquad\qquad +\frac{\Gamma(M-1/2)}{\Gamma(M)} \frac{\Gamma(2M-n)}{\Gamma(2M-1)} - \frac{\Gamma(M+1/2)}{\Gamma(M+1)} \frac{\Gamma(2M-n+)}{\Gamma(2M)}\\
&=\frac{\Gamma(M+1/2)}{\Gamma(M+1)} \frac{\Gamma(2M-n+2)}{\Gamma(2M)} \left(\frac 1{2M}-\frac 1{2M-n+1}\right)+\\
&\qquad\qquad\qquad\qquad\qquad + \frac{\Gamma(M+1/2)}{\Gamma(M+1)} \frac{\Gamma(2M-n+1)}{\Gamma(2M)} \left(\frac {2M}{2M-n}-1\right)\\
&=\frac{\Gamma(M+1/2)}{\Gamma(M+1)} \frac{\Gamma(2M-n)}{\Gamma(2M+1)} \bigl(2M+n(n-1)\bigr)\ .
\end{align*}
Using the duplication formula \eqref{duplication} for $z=M-n$ and $z=M+1$ we get finally
\begin{align*}
 S&=\frac{\Gamma(M+1/2)}{\Gamma(M+1)} \frac{2^{2M-n-1} \Gamma(M-\frac n2) \Gamma(M-\frac{n-1}2)}{2^{2M} \Gamma(M+1/2)\Gamma(M+1)} \bigl(2M+n(n-1)\bigr)\\
 &=\frac 1{2^{n}} \frac{\Gamma(M-\frac n2) \Gamma(M-\frac{n-1}2)} {(\Gamma(M+1))^2}\bigl(M+\frac{n(n-1)}2\bigr)\ .
\end{align*}

\end{proof}


\begin{bibdiv}

\begin{biblist}


\bib{Bu13}{article}{
   author={Buckley, Jeremiah},
   title={Fluctuations in the zero set of the hyperbolic Gaussian analytic function},
   journal={Int. Math. Res. Not. IMRN to appear},
   volume={},
   date={2013},
   number={},
   pages={18},
   doi={},
}

\bib{BMP}{article}{
   author={Buckley, Jeremiah},
   author={Massaneda, Xavier},
   author={Pridhnani, Bharti},
   title={Gaussian analytic functions in the ball},
   journal={Preprint},
   date={2014},
   number={},
   pages={},
   issn={},
   review={},
}

\bib{HKPV}{book}{
   author={Hough, John Ben},
   author={Krishnapur, Manjunath},
   author={Peres, Yuval},
   author={Vir{\'a}g, B{\'a}lint},
   title={Zeros of Gaussian analytic functions and determinantal point
   processes},
   series={University Lecture Series},
   volume={51},
   publisher={American Mathematical Society},
   place={Providence, RI},
   date={2009},
   pages={x+154},
   isbn={978-0-8218-4373-4},
   review={\MR{2552864 (2011f:60090)}},
}






\bib{NS}{article}{
   author={Nazarov, Fedor},
   author={Sodin, Mikhail},
   title={Fluctuations in random complex zeroes: asymptotic normality
   revisited},
   journal={Int. Math. Res. Not. IMRN},
   date={2011},
   number={24},
   pages={5720--5759},
   issn={1073-7928},
   review={\MR{2863379 (2012k:60103)}},
}


\bib{Ru}{book}{
   author={Rudin, Walter},
   title={Function theory in the unit ball of $\mathbb C^n$},
   series={Classics in Mathematics},
   note={Reprint of the 1980 edition},
   publisher={Springer-Verlag},
   place={Berlin},
   date={2008},
   pages={xiv+436},
   isbn={978-3-540-68272-1},
   review={\MR{2446682 (2009g:32001)}},
}

\bib{SZ06}{article}{
   author={Shiffman, Bernard},
   author={Zelditch, Steve},
   title={Number variance of random zeros},
   journal={ArXiv:\newline  arxiv.org/pdf/math/0608743v3},
   volume={},
   date={2006},
   number={},
   pages={},
   issn={},
   review={},
   doi={},
}

\bib{SZ08}{article}{
   author={Shiffman, Bernard},
   author={Zelditch, Steve},
   title={Number variance of random zeros on complex manifolds},
   journal={Geom. Funct. Anal.},
   volume={18},
   date={2008},
   number={4},
   pages={1422--1475},
   issn={1016-443X},
   review={\MR{2465693 (2009k:32019)}},
   doi={10.1007/s00039-008-0686-3},
}


\bib{Sod}{article}{
   author={Sodin, Mikhail},
   title={Zeros of Gaussian analytic functions},
   journal={Math. Res. Lett.},
   volume={7},
   date={2000},
   number={4},
   pages={371--381},
   issn={1073-2780},
   review={\MR{1783614 (2002d:32030)}},
}

\bib{ST1}{article}{
   author={Sodin, Mikhail},
   author={Tsirelson, Boris},
   title={Random complex zeroes. I. Asymptotic normality},
   journal={Israel J. Math.},
   volume={144},
   date={2004},
   pages={125--149},
   issn={0021-2172},
   review={\MR{2121537 (2005k:60079)}},
   doi={10.1007/BF02984409},
}

\bib{Stoll}{book}{
   author={Stoll, Manfred},
   title={Invariant potential theory in the unit ball of ${\bf C}^n$},
   series={London Mathematical Society Lecture Note Series},
   volume={199},
   publisher={Cambridge University Press},
   place={Cambridge},
   date={1994},
   pages={x+173},
   isbn={0-521-46830-2},
   review={\MR{1297545 (96f:31011)}},
   doi={10.1017/CBO9780511526183},
}

\end{biblist}
\end{bibdiv}
\end{document}